\documentclass[12pt]{article}
\usepackage{graphicx, amsmath,amsfonts,amsthm, euscript}
\usepackage{verbatim}
\usepackage{color}
\usepackage{ulem}
\usepackage{amsmath}
\usepackage{amscd}
\usepackage{amsfonts}
\usepackage{braket,amsfonts}
\usepackage{array}
\usepackage[caption=false]{subfig}
\usepackage{pgfplots}
\usepackage{graphicx,epstopdf}
\usepackage{hyperref}
\usepackage{hyperref}
\hypersetup{colorlinks,allcolors=black}

\usepackage{graphicx}
\usepackage {tikz}
\usetikzlibrary {positioning}
\usetikzlibrary{shapes,arrows}
\usepackage{appendix}
\usepackage{lineno}
\newtheorem{theorem}{Theorem}[section]
\newtheorem{proposition}{Proposition}[section]

\newtheorem{lemma}{Lemma}[section]
\newtheorem{corollary}{Corollary}[section]
\newtheorem{remark}{Remark}[section]

 \numberwithin{equation}{section}
  \numberwithin{figure}{section}
\usepackage{graphicx, amsmath}
\usepackage{color}
\definecolor{Blue}{rgb}{0.3,0.3,0.9}

\newtheorem*{Ack}{Acknowledgment}
\begin{document}

\normalsize

\title{Additive estimates of the permanent using Gaussian fields}

\maketitle

\vskip 6pt
\author{ Tantrik Mukerji and  Wei-Shih Yang}
\vskip 3pt
\indent %Department of Mathematics\\
 \indent        %Temple University,\\
 \indent  %Wachman Hall, \\
  \indent %1805 North Broad Street, \\
   \indent     %Philadelphia, PA 19122

\vskip 6pt
Email: tuj79252@temple.edu, yang@temple.edu \\

\vskip 12pt


\vskip12pt
%AMS classification Primary:   15A15, 60G15, 68W20 

\begin{abstract}
We present a randomized algorithm for estimating the permanent of an $M \times M$ real matrix $A$ up to an additive error. We do this by viewing the permanent  $\mathrm{perm}(A)$ of $A$ as the expectation of a product of centered joint Gaussian random variables with a particular covariance matrix $C$. 
The algorithm outputs the empirical mean $S_{N}$ of this product after sampling $N$ times. Our algorithm runs in total time $O(M^{3} + M^{2}N + MN)$ with failure probability 
\begin{equation*}
    P(|S_{N}-\text{perm}(A)| > t) \leq \frac{3^{M}}{t^{2}N} \prod^{2M}_{i=1} C_{ii}.
\end{equation*}
In particular, we can estimate $\mathrm{perm}(A)$ to an additive error of $\epsilon\bigg(\sqrt{3^{2M}\prod^{2M}_{i=1} C_{ii}}\bigg)$ in polynomial time. We compare to a previous procedure due to Gurvits.  We discuss how to find a particular $C$ using a semidefinite program and a relation to the Max-Cut problem and cut-norms.
\end{abstract}

\newpage

\newpage

%%%%%%%%%%%%%%%%%%%%%%%%%%%%%%%%%%-Begin-Body-%%%%%%%%%%%%%%
\section{Introduction}
\setcounter{equation}{0}

The permanent of an  $M \times M$ matrix $A = (A_{ij}) $ is defined as
\begin{equation}
    \mathrm{perm}(A) = \sum_{\sigma} \prod_{i} A_{i \sigma(i)},
\end{equation}
\newline
where the sum is over all permutations $\sigma$ of $\{1, \ldots, M \}$.   The computation of the permanent often occurs in the context of counting. An important example is when the entries of $A$ are elements of the set $\{0,1\}$. In this case, we can associate $A$ to a bipartite graph and  $\text{perm}(A)$ then counts the number of perfect matchings of this bipartite graph. The number of perfect matchings of a graph has applications in a variety of contexts related to combinatorial optimization, with applications to flow networks, scheduling, modeling bonds in chemistry, etc. 
\newline \indent Despite having a similar definition to the determinant, the computation of $\text{perm}(A)$ is believed to be difficult. Formally, it was shown by Valiant in \cite{V} that computing the permanent of a matrix with entries in $\{0,1\}$ is $\#{P}$-complete, the representative class of the hardest problems in the class $\#{P}$. Recently, the difficulty of computing the permanent was used in the advent of boson sampling by Aaronson-Arkhipov in \cite{AA}. There the authors introduce a model of computing based on linear optics. The sampling task involved with this model is believed to be classically intractible since the probability distribution from which the model samples from is related to the permanent of a complex matrix. Furthermore, the framework of linear optics has been used by Aaronson in \cite{A} to provide another proof showing that computing $\mathrm{perm}(A)$ is $\#{P}$-hard and furthermore showed that even computing the sign of $\mathrm{perm}(A)$ is $\#{P}$-hard. With this, recent work has been focused on obtaining an approximation for $\mathrm{perm}(A)$ especially for particular classes of matrices where the problem may be more tractible.
\newline \indent 
For obtaining a multiplicative (relative) approximation of a matrix $A$ with non-negative entries, it was seminally shown by Jerrum, Sinclair, and Vigoda in \cite{JSV} that there exists a randomized algorithm that outputs $Z$ such that $$P[\exp(-\epsilon)Z \leq \text{perm}(A) \leq \exp(\epsilon)Z] \geq \frac{3}{4}$$ in time polynomial in $M$ and $1/\epsilon$ with a running time of $\tilde{O}(M^{10})$ using warm-starts. For positive semidefinite matrix $A$, the earliest multiplicative approximation result is due to Marcus \cite{Ma} which provides an approximation within a factor of $M!$ using the diagonal entries of the matrices.  Recent works \cite{AG} and \cite{YP}, provide a polynomial time deterministic approximation within an exponential factor $c^{M}$ for $c = e^{1+\gamma}$, with $\gamma$ being Euler's constant. It was recently shown by Meiburg in \cite{M}, that approximating the permanent of a positive semidefinite matrix within an exponential factor $2^{M^{1-\epsilon}}$ is NP-hard for any $\epsilon > 0$.
\newline \indent
For obtaining an additive estimate of $\mathrm{perm}(A)$. Gurvits showed in \cite{G} the existence of randomized polynomial time algorithm that estimates $\mathrm{perm}(A)$, for any matrix $A$, within an additive error given in terms of $\|A\|$. Here $\|A\|$ is the operator norm of $A$. In particular, Gurvits showed the following whose presentation we take from \cite{AH}.
\begin{theorem} \label{Gurvits}
There is a randomized algorithm such that for any $M \times M$ complex matrix $A$ returns an empirical sample $S_{N}$ on $N$ samples from a uniformly distributed bit string $\{-1, 1\}$, which can be computed in $O(M^{2}N)$ time, whose additive estimate for $\text{perm}(A)$ has the upper bound on the failure probability
$$P[|S_{N} - Per(A)| > t] \leq 2\exp(-Nt^{2}/2\|A \|^{2M}).$$
In particular, if we take $\frac{K}{\epsilon^{2}}$ samples then this algorithm in time $O(M^{2}/\epsilon^{2})$, obtains an estimate of $\mathrm{perm}(A)$ within an additive error of $\epsilon\|A\|^{M}$ with probability at least $1-  2\exp(-K/2).$
\end{theorem}
The work does this by viewing $\mathrm{perm}(A)$ as the expectation of a bounded random variable on uniformly distributed bit-strings $\{-1,1\}^{M}$. The algorithm then proceeds by outputting the empirical mean $S_{N}$ of $N$ samples of this random variable which can be computed at a cost of $O(M^{2})$ per sample and taking $N$ to be $O(\frac{1}{\epsilon^{2}})$. Hoeffding’s inequality implies that the upper bound on the failure probability. 
\newline \indent 
Another additive estimate of $\text{perm}(A)$ for the case when $A$ is Hermitian positive semidefinite was provided by the work of Chakhmakhchyan, Cerf, and Garcia-Patron in \cite{CCGP}. There they exploited a physical interpretation of $\text{perm}(A)$ within linear optics when the input is a thermal state. There $A$ can be associated to a unitary transformation on the state space and $\text{perm}(A)$ is proportional to the probability of detecting a single photon in each mode at the output. The authors use this to show that $\text{perm}(A)$ is proportional to the expectation of a bounded random variable over a complex Gaussian distribution. The algorithm then approximates $\text{perm}(A)$ by taking the empirical mean of this random variable which also takes $O(M^{2})$ time to compute for each of the samples. This procedure involves diagonalizing $A$ which comes at an additional computational cost of $O(M^{3})$ but improves on the error of Gurvits’s algorithm if the eigenvalues of $A$ satisfy a certain smoothness property. This algorithm matches the computation time of our algorithm below. 
\newline \indent In this work, we construct another randomized polynomial time additive estimate of the permanent $\text{perm}(A)$ for any $M \times M$ real matrix $A$ using samples from a Gaussian distribution.
\newline \indent
When $A$ is symmetric positive definite, one can associate to $A$, the Gaussian random variables $X_{1}, \ldots, X_{M}$ whose covariance is given by $A$. The random variables $X_{1}, \ldots, X_{M}$ are referred to as the \textit{Gaussian field} associated to $A$. These random variables span a Hilbert space, whose associated Fock space has a multiplication structure known as the $\textit{Wick product}$ (see \cite{J}). The relationship of the Wick product to the permanent of $A$ is given by the following theorem that has appeared implicitly in literature (see chapter 3 of \cite{J}).

\begin{theorem} \label{theorem1.2}
Let $X_{1}, \ldots, X_{M}$ be jointly Gaussian with real positive-definite symmetric covariance matrix given by $A_{ij} = \langle X_{i}X_{j} \rangle$ and let $:X_{1}\ldots{X_{M}}:$ denote the \textit{Wick product} of  the $X_{i}$, we then have that 
$$ \langle :X_{1}\ldots{X_{M}}:^{2} \rangle =\mathrm{perm}(A).$$
Here $\langle \cdot \rangle$ denotes expectation.
\end{theorem}

In order to make our treatment self-contained, we will prove this theorem in the sequel using the combinatorial framework provided by Dynkin in \cite{D}. 
\newline \indent We then use the techniques of the proof of Theorem 1.1 to relate $\text{perm}(A)$, for any $M \times M$ real matrix  $A$, to the expectation of a product of a  Gaussian field $X_{1}, \ldots, X_{M'}$ whose covariance matrix contains $A$ as a submatrix. We refer to the covariance matrix $C$ of $X_{1}, \ldots, X_{M'}$ as a $\textit{Gaussian embedding}$ of $A$. We denote by $\mathcal{N}(0, C)$ to be the centered Gaussian distribution with covariance matrix $C$.  We provide a proof of the result that follows.

\begin{theorem}  \label{theorem1.3}
Let $A$ be an arbitrary $M \times M$ real matrix. There exists a $2M \times 2M$ covariance matrix $C$, such that the associated Gaussian field $X_{1}, \ldots, X_{2M}$ has a product that satisfies
\begin{equation}
    \mathrm{perm}(A) = \big\langle \prod^{2M}_{j=1} X_{j} \big\rangle.
\end{equation}
Here $\langle \cdot \rangle$ denotes expectation.
\end{theorem}

We now present our main contributions starting with the following result.

\begin{theorem} \label{theorem1.4}
Let $A$ be an $M \times M$ real matrix and let $C$ be a $2M \times 2M$ positive semidefinite symmetric matrix with off-diagonal entries, $i \neq j$, given by $$
C_{ij} =
\begin{cases}
A_{ij} & 1\leq i \leq M, M+1 \leq j \leq 2M \\
0&\text{otherwise}
\end{cases}$$ and diagonal entries given by $C_{ii} = \alpha$ for any $ \alpha \geq \|A\|$. 
Then given $N$ samples from $\mathcal{N}(0,C)$, there is a randomized algorithm that returns an empirical sample $S_{N}$  which can be computed in time $O(MN)$, whose additive estimate for $\text{perm}(A)$ has the failure probability:
\begin{equation}
    P(|S_{N}-\mathrm{perm}(A)| > t) \leq \frac{3^{M}}{t^{2}N}{\alpha}^{2M}
\end{equation}
 
In particular, if we are given $\frac{K}{\epsilon^{2}}$ samples then this algorithm in time $O(M/\epsilon^{2})$ obtains an estimate of $\mathrm{perm}(A)$ within an additive error of $\epsilon(\sqrt{3}\|A\|)^{M}$ with probability at least $1- K^{-1}.$

\end{theorem}

 \indent 
This result can be compared to theorem \ref{Gurvits} but can improved since there is a degree of flexibility to our approach. Namely, the upper bound on the failure probability in theorem \ref{theorem1.4} is provided by using independence within the Gaussian field. This argument can be generalized to provide a stronger upper bound on the failure probability. In particular, the proof of theorem \ref{theorem1.4} allows us to state the upper bound of the failure probability in terms of the diagonal entries of the covariance matrix $C$. We give this restatement as follows.
%Let $A$ b
\begin{theorem} \label{theorem1.5}
   Let $A$ be an $M \times M$ real matrix and let $C$ be a $2M \times 2M$ positive semidefinite symmetric matrix with off-diagonal entries, $i \neq j$, given by $$
C_{ij} =
\begin{cases}
A_{ij} & 1\leq i \leq M, M+1 \leq j \leq 2M \\
0&\text{otherwise}
\end{cases}.$$
Then given $N$ samples from $\mathcal{N}(0,C)$, there is a randomized algorithm that returns an empirical sample $S_{N}$  which can be computed in time $O(MN)$, whose additive estimate for $\text{perm}(A)$ has the failure probability:
\begin{equation}
    P(|S_{N}-\mathrm{perm}(A)| > t) \leq \frac{3^{M}}{t^{2}N} \prod^{2M}_{i=1} C_{ii}.
\end{equation}
In particular, if we are given $\frac{K}{\epsilon^{2}}$ samples then this algorithm in time $O(M/\epsilon^{2})$ obtains an estimate of $\mathrm{perm}(A)$ within an additive error of $\epsilon\bigg(\sqrt{3^{2M}\prod^{2M}_{i=1} C_{ii}}\bigg)$ with probability at least $1- K^{-1}.$
\end{theorem}
This restatement allows us to vary the diagonal entries of $C$ to improve the upper bound on the failure probability. We consider discuss finding such a $C$ through the use of a semidefinite program discussed below. In the case, when $A_{ij} \geq 0$ for all $i,j$, we consider the case where $C$ is a diagonally dominant matrix and show it improve over the covariance matrix used in theorem \ref{theorem1.4} and furthermore show it to be the optimal solution for the semidefinite program we consider. We also show a relation between the semidefinite program we consider and the semidefinite program that appeared in \cite{GW} for the Max-Cut problem for arbitrary graphs with non-negative weights.  In the general case, when the entries of $A$ are potentially negative, we show a connection between the semidefinite program to the Max-Cut problem of real weighted bipartite graphs and the cut norms that appeared in the work of Alon and Naor \cite{AN}. We are able to bound the failure probability of our procedure in terms of these norms in corollary \ref{lastcor1} below. 
\newline \indent
In comparison to theorem \ref{Gurvits}, we sample from $\mathcal{N}(0, C)$ as opposed to $\{-1, 1\}^{M}$.  To obtain $N$ samples, after an initial cost of $O(M^{3})$ to perform a Cholesky decomposition on $C$, obtains the samples in $O(M^{2}N)$ time and computes the estimate in time $O(MN)$.   In addition, the total computation time of our procedure may include the cost of finding $C$ and obtaining $N$ samples from $\mathcal{N}(0,C)$ which we comment on below for various cases. In addition, the bounds for the failure probability of our procedure as stated in theorem \ref{theorem1.4} is exponentially weaker than the procedure of theorem \ref{Gurvits} since our procedure estimates estimates $\mathrm{perm}(A)$, for any real matrix $A$, within an additive error of $\epsilon(\sqrt{3}\|A\|)^{M}$ rather than $\|A\|^{M}$. However, the improvement provided by theorem \ref{theorem1.5} allows us to find classes of matrices where our procedure potentially improves upon the accuracy of the procedure provided by theorem \ref{Gurvits}. In particular, when  $3^{2M}\prod^{2M}_{i=1} C_{ii} \leq \|A\|^{2M}$ then our procedure has a stronger accuracy guarantee then the one provided by theorem \ref{Gurvits}. 
\newline \indent 
\indent We organize this work by first providing a background section that contains the combinatorial framework for expectations of Wick products described in \cite{D} which we then use to provide a proof of theorem \ref{theorem1.2}. We then define what it means to be a Gaussian embedding and prove theorem \ref{theorem1.3}. We can then provide the algorithm in theorem \ref{theorem1.4} and its subsequent analysis. We then consider the relaxation of our procedure provided by theorem \ref{theorem1.5} and optimize over all possible choices of diagonal matrices via a semidefinite program. We then discuss the case when $A$ has non-negative entries and the relationship between our program and the Max-Cut semidefinite program in \cite{GW}. We then discuss a connection of our program to the Max-Cut problem on real weighted bipartite graphs and the cut norms on $A$ discussed in \cite{AN}. We hope that this work opens the door to further connections between Gaussian fields and the permanent.

\section{Background}  \indent Our treatment here follows Dynkin in \cite{D}. Let $(\Omega, \mathcal{F}, P)$ be a probability space and consider the Hilbert space $L^{2}(\Omega, \mathcal{F}, P)$. Given a random variable $X: \Omega \rightarrow \mathbb{R}$, we denote its expectation as $\langle X \rangle$. Furthermore, we say a random variable is \textit{centered} if $\langle X \rangle  = 0$. A finite collection of centered random variables $X_{1}, \ldots, X_{n}$ is said to be \textit{jointly Gaussian} if for any $t_{i} \in \mathbb{R}$, we have that $\langle \exp(t_{1}X_{1} +  \cdots t_{n}X_{n}) \rangle = \exp( \frac{1}{2}\sum_{i,j} A_{ij}t_{i}t_{j})$ where $A_{ij} = \langle X_{i} X_{j} \rangle$ is the covariance between $X_{i}$ and $X_{j}$. We refer to $X_{1}, \ldots, X_{n}$ as the centered \textit{Gaussian field} with covariance matrix $A$ and refer to the linear span $H$ of $X_{1}, \ldots, X_{n}$ as the associated $\textit{Gaussian space}$. 
 \newline \indent
 For the expectation of a product $X_{i_{1}}\ldots{X_{i_{p}}}$, we have Wick's theorem or Isserlis' theorem (see Theorem 1.28 in \cite{J}).

 We now define the Wick product on a Gaussian space $H$ spanned by a finite number of centered Gaussian random variables.  Given $Y \in H$, define the \textit{Wick product} $: \cdot :$ on the exponential of $Y$ as 
 \begin{equation} \label{eq2.2}
     :e^{Y}: = e^{{Y} - \frac{\langle Y^{2}\rangle}{2} }
 \end{equation} 
 We also define the Wick product  on products of the form $X^{k_{1}}_{1}\ldots{X^{k_{n}}_{n}}$ via generating functions as
  \begin{equation} \label{Wickprod1}
     \sum^{\infty}_{k_{1}, \ldots, k_{n}} \frac{t_{1}^{k_{1}}\cdots{t_{n}^{k_{n}}}}{k_{1}!\cdots{k_{n}!}}:X^{k_{1}}_{1}\cdots{X^{k_{n}}_{n}}: = :e^{\sum^{n}_{i=1} t_{i}X_{i}}:
 \end{equation} 
 where \eqref{eq2.2} applies to the right hand side which simplifies to
   \begin{equation} \label{Wickprod2}
     \sum^{\infty}_{k_{1}, \ldots, k_{n}} \frac{t_{1}^{k_{1}}\cdots{t_{n}^{k_{n}}}}{k_{1}!\cdots{k_{n}!}}:X^{k_{1}}_{1}\cdots{X^{k_{n}}_{n}}: {}= \exp({\sum^{n}_{i=1}t_{i}X_{i} -\frac{1}{2} \sum^{n}_{i,j}  A_{ij}t_{i}t_{j}})
    \end{equation} 
where $A_{ij} = \langle X_{i}X_{j} \rangle$.
\newline \indent
 It then follows from Taylor's theorem that $:X^{k_{1}}_{1}\cdots{X^{k_{n}}_{n}}:$ is the Taylor coefficient of $t_{1}^{k_{1}}\cdots{t_{n}^{k_{n}}}$. That is
 \begin{equation} \label{eq2.5}
     :X^{k_{1}}_{1}\cdots{X^{k_{n}}_{n}}: {} = {} \frac{\partial^{k_{1} + \cdots + k_{n}}}{\partial{t}^{k_{n}}_{n}\ldots\partial{t_{1}^{k_{1}}}} \exp( \sum^{n}_{i=1}t_{i}X_{i} -\frac{1}{2} \sum^{n}_{i,j=1} A_{ij}t_{i}t_{j}) \bigg \vert_{t_{i}=0}.
 \end{equation}
It is well known that the $:X^{k_{1}}_{1}\cdots{X^{k_{n}}_{n}}:$ form a basis for $L^{2}(\Omega, \mathcal{F}(H), P)$ where $\mathcal{F}(H)$ denotes the $\sigma$-algebra generated by $H$, (see for example, Theorem 2.6 (pg 18) of \cite{J}) which allows us to define the Wick product on $L^{2}(\Omega, \mathcal{F}(H), P)$.
\newline \indent
Let $S$ be a finite subset of $H$ and let $\mathcal{P} = \{S_{1}, \ldots, S_{r} \}$ be a partition of $S$ where each element $S_{i}$ of the partition can be labeled as $S_{i} = \{X_{i1}, \ldots, X_{i{n_{i}}} \}$. Here the subscript $ij$ indicates the variable being the $j$th element of $S_{i}$ i.e. $X_{ij}  \in S_{i}$. This notation will be useful for grouping products of Wick products, for which we have the following proposition that appears as equation 2.11 in \cite{D}.
\begin{proposition} \label{prop2.2} 

Let $X_{ij}$ be centered Gaussian random variables for $i=1,\ldots, r$ and  $j = 1, \ldots, n_{i}$. Then the following holds
\begin{equation}
    \bigg\langle \prod^{r}_{i=1} : \prod^{n_{i}}_{j=1} X_{ij}:  \bigg\rangle  =   \sum \prod \langle X_{i_{\alpha}j_{\alpha}}X_{i_{\beta}j_{\beta}} \rangle 
\end{equation}
where the sum is taken over all partitions of $\{(i,j)\}_{1 \leq i \leq r, 1 \leq j \leq n_{i}}$ into pairs $((i_{\alpha}, j_{\alpha}), (i_{\beta},j_{\beta})) $ where $i_{\alpha} \neq i_{\beta}$. When no such partition exists, then we have that the right-hand side is zero.
\end{proposition}
As consequence of proposition \ref{prop2.2}, one can obtain Wick's theorem.

\begin{corollary}
    \label{Wickthm}
 For the centered Gaussian random variables $X_{i_{1}}, \ldots, X_{i_{p}}$ where $i_{k} \in \{1, \ldots, m \}$, we have that
\begin{equation} \label{eq2.1}
      \langle X_{i_{1}}\ldots{X_{i_{p}}} \rangle  = \begin{cases} 0 &\text{if $p$ is odd} 
  \\
  \sum_{\text{pairs}} \langle X_{i_{\alpha_{1}}}X_{i_{\beta{1}}} \rangle \cdots \langle X_{i_{\alpha_{k}}}X_{i_{\beta{k}}} \rangle 
  &\text{if $p =2k$}
  \end{cases}
\end{equation}
  where the sum is taken over all partitions of $\{1, \ldots, 2k \}$ into unordered pairs $(\alpha_{i}, \beta_{i})$. 
\end{corollary}

We present a proof of  proposition following the outline in \cite{D}. By \eqref{eq2.5},
we have that 
\begin{equation}
    : \prod^{n_{i}}_{j=1} X_{ij}: = \frac{\partial^{n_{i}}}{\partial{t}_{i,n_{i}}\ldots\partial{t}_{i,1}} \exp( \sum^{n_{i}}_{j=1}t_{ij}X_{ij} -\frac{1}{2} \sum^{n_{i}}_{j,k=1} \langle X_{ij}X_{ik} \rangle t_{ij}t_{ik}) \bigg \vert_{t_{i}=0}.
\end{equation}

We have in particular, that 
\begin{equation} \label{eq2.8}
    : \prod^{n_{i}}_{j=1} X_{ij}: = \kappa(\exp( \sum^{n_{i}}_{j=1}t_{ij}X_{ij} -\frac{1}{2} \sum^{n_{i}}_{j,k=1} \langle X_{ij}X_{ik} \rangle t_{ij}t_{ik}))
\end{equation}
where $\kappa(f(t_{1}, \ldots, t_{n}))$ is the coefficient of $t_{1}\ldots{t_{n}}$ in the Taylor series expansion of $f(t_{1}, \ldots, t_{n})$. This brings us to the following lemma needed in the proof of Proposition \ref{prop2.2}.
\begin{lemma} \label{lem2.1}
With the notation the same as Proposition \ref{prop2.2}, we have that 
   $$   \prod^{r}_{i=1} \kappa(\exp( \sum^{n_{i}}_{j=1}t_{ij}X_{ij} -\frac{1}{2} \sum^{n_{i}}_{j,k=1} \langle X_{ij}X_{ik} \rangle t_{ij}t_{ik})) = \kappa \bigg(\prod^{r}_{i=1} \exp( \sum^{n_{i}}_{j=1}t_{ij}X_{ij} -\frac{1}{2} \sum^{n_{i}}_{j,k=1} \langle X_{ij}X_{ik} \rangle t_{ij}t_{ik}) \bigg) $$
\end{lemma}
\begin{proof}
This follows from \eqref{Wickprod1} and comparing Taylor coefficients.
\end{proof}
We now present a proof of proposition \ref{prop2.2}.
\begin{proof}
From the definition of $\kappa$ \eqref{eq2.8}, Lemma \ref{lem2.1}, and linearity of expectation we have that
\newline
\begin{align*}
         \bigg\langle \prod^{r}_{i=1} : \prod^{n_{i}}_{j=1} X_{ij}:  \bigg\rangle &=  \bigg\langle   \prod^{r}_{i=1} \kappa(\exp( \sum^{n_{i}}_{j=1}t_{ij}X_{ij} -\frac{1}{2} \sum^{n_{i}}_{j,k=1} \langle X_{ij}X_{ik} \rangle t_{ij}t_{ik}) \bigg\rangle \\
     &= \kappa \bigg( \bigg\langle \prod^{r}_{i=1} \exp( \sum^{n_{i}}_{j=1}t_{ij}X_{ij} -\frac{1}{2} \sum^{n_{i}}_{j,k=1} \langle X_{ij}X_{ik} \rangle t_{ij}t_{ik}) \bigg\rangle \bigg).
\end{align*}
We then use \eqref{Wickprod2} to get
$$\kappa \bigg( \bigg\langle \prod^{r}_{i=1} \exp( \sum^{n_{i}}_{j=1}t_{ij}X_{ij} -\frac{1}{2} \sum^{n_{i}}_{j,k=1} \langle X_{ij}X_{ik} \rangle t_{ij}t_{ik}) \bigg\rangle \bigg) = \kappa \bigg( \bigg\langle \prod^{r}_{i=1}:\exp(\sum^{n_{i}}_{j=1}t_{ij}X_{ij}):\bigg\rangle\bigg),$$
we can then write 
$$\bigg\langle \prod^{r}_{i=1}:\exp(\sum^{n_{i}}_{j=1}t_{ij}X_{ij}):\bigg\rangle = \bigg\langle  :\exp(\sum^{r}_{i=1} \sum_{j=1}^{n_{i}}t_{ij}X_{ij}):\exp(\sum_{i < i', j,j'}t_{ij}t_{i'j'} \langle X_{ij}X_{i'j'} \rangle)\bigg\rangle$$
using \eqref{eq2.2}. 
\newline \indent 
Since the term $\exp(\sum_{i < i', j,j'} \langle X_{ij}X_{i'j'} \rangle)$ is not random and $$\bigg\langle  :\exp(\sum^{r}_{i=1} \sum_{j=1}^{n_{i}}t_{ij}X_{ij}): \bigg\rangle = 1,$$  we have that 
$$\kappa \bigg( \bigg\langle \prod^{r}_{i=1} \exp( \sum^{n_{i}}_{j=1}t_{ij}X_{ij} -\frac{1}{2} \sum^{n_{i}}_{j,k=1} \langle X_{ij}X_{ik} \rangle t_{ij}t_{ik}) \bigg\rangle \bigg)  = \kappa(\exp(\sum_{i < i', j,j'} t_{ij}t_{i'j'}\langle X_{ij}X_{i'j'} \rangle)).$$
We then obtain the claim by the definition of $\kappa$ \eqref{eq2.8} and the Faà di Bruno's formula (\cite{C}).

\end{proof}

In the case where the sum of $n_{i}$ are even, the expectation in proposition \ref{prop2.2} can be described graphically by associating to each $:\prod^{n_{i}}_{i=1}X_{i} :$ for $1 \leq i \leq r$, a vertex ${i}$ with $n_{i}$ legs each labeled $(i,j)$ for $1 \leq j \leq n_{i}$. A \textit{complete Feynman diagram} $\gamma$ on the a set of Wick products $\{:\prod^{n_{i}}_{j=1}X_{j} : | 1 \leq  i \leq r\}$ is obtained when each leg $(i,j)$ is paired to a leg  $(i',j')$ from a different vertex $i'$ associated to $:\prod^{n_{i}}_{j'=1}X_{j'} :$ , the associated value $\nu(\gamma)$ of the diagram $\gamma$ is defined to be the product $\nu(\gamma) = \prod \langle X_{i_{\alpha}j_{\alpha}}X_{i_{\beta}j_{\beta}} \rangle$  where the product is taken over each pairing of legs $((i_{\alpha}, j_{\alpha}), (i_{\beta},j_{\beta}))$ from different vertices in the diagram $\gamma$. In particular, it follows that we can restate proposition \ref{prop2.2} as the following proposition that appears as equation (2.12) in \cite{D}.

\begin{proposition} \label{prop2.3}
  Let $X_{ij}$ be centered Gaussian random variables for $i=1,\ldots, r$ and  $j = 1, \ldots, n_{i}$. Then 
\begin{equation} 
\bigg\langle \prod^{r}_{i=1} : \prod^{n_{i}}_{j=1} X_{ij}:  \bigg\rangle  =  \sum_{\gamma \in \Gamma}  \nu(\gamma) 
\end{equation}
where the sum is taken over the set all complete Feynman diagrams $\Gamma$ over the set of Wick products $ :\prod^{n_{i}}_{j=1} X_{j}: $. When $\Gamma$ is empty, then the right hand side is zero.
\end{proposition}

A consequence of this is the following theorem

   \begin{theorem} \label{theorem2.1}
   Let $X_{1j}$ and $X_{2j}$ for $ 1 \leq j \leq n$ be in $H$, then
$$ \langle :X_{11} \ldots X_{1n}::X_{21}\ldots{X_{2n}}: \rangle  = \sum_{ \sigma \in S_{n}} \prod^{n}_{i=1} \langle X_{1i}X_{2\sigma(i)} \rangle$$
   \end{theorem}
   We can then use this result to get an immediate connection to the permanent in the following corollary.
   
   \begin{corollary} \label{permWick}
   Let $X_{1}, \ldots, X_{n}$ be jointly Gaussian with real positive-definite symmetric covariance matrix given by $A_{ij} = \langle X_{i}X_{j} \rangle$, we then have that 
   $$ \langle :X_{1}\ldots{X_{n}}:^{2} \rangle = Perm(A).$$
   
   \end{corollary}
   
\begin{remark}
The connection made to the permanent here has appeared before in the literature. See for instance pg 26 of \cite{J} for a different treatment of Wick products where Theorem 2.1 appears. For a similar treatment of Wick products, see chapter 2 of Simon \cite{S}.
\end{remark}

Corollary \ref{permWick} provides intuition for a procedure to provide an additive estimate for the permanent by the law of large numbers.

\section{The Algorithm and Analysis}

\subsection{Gaussian embedding of an arbitrary matrix $A$}
\indent
Let $C$ be a real symmetric positive semidefinite $M \times M$ matrix.  We can associate $C$ to the  centered Gaussian field $(X_{1}, \ldots, X_{M})$ of joint Gaussian random variables $X_{1}, \ldots X_{M}$ with covariance given by $\langle X_{i}X_{j} \rangle = C_{ij}$. Given a subset $\Gamma = \{\gamma_{1}, \ldots, \gamma_{k} \}$ of $\{1, \ldots, M\}$, we have the associated \textit{Gaussian subfield} $X_{\Gamma} := (X_{\gamma_{1}}, \ldots, X_{\gamma_{k}})$ by the entries of the associated submatrix of $C$. Using this terminology we can show the following result which is similar to theorem \ref{theorem2.1}.
\begin{proposition} Let $S = \{s_{1}, \ldots, s_{M} \}$ and $T= \{t_{1}, t_{2}, \ldots, t_{M} \}$ be subsets of $\{1, \ldots, 2M \}$. We then define the matrix $C_{S,T}$ whose $ij$th entries are given by $(C_{ST})_{ij} = C_{s_{i},t_{j}}$. Furthermore, let $X_{S}$ and $X_{T}$ be the associated Gaussian subfields to $S$ and $T$ respectively, then
$$\mathrm{perm}(C_{S,T}) = \langle : \prod^{M}_{i=1} X_{s_{i}}:: \prod^{M}_{i=1} X_{t_{i}}: \rangle.$$
\end{proposition}
\begin{proof}
We can associate each product $: \prod^{M}_{i=1} X_{s_{i}}:$ and $: \prod^{M}_{i=1} X_{t_{i}}: $ to two distinct vertices $v_1$ and $v_{2}$ with $M$ legs each, the legs $(1,i)$ for $v_{1}$ are associated to the random variables $X_{s_{i}}$ and likewise for $v_{2}$. We then have that the claim follows from proposition \ref{prop2.3}.
\end{proof}
\indent We say a real $M \times M$ matrix $A$ is \textit{$M'$-Gaussian embeddable} if there exits a Gaussian field $(X_{1}, \ldots, X_{M'})$ such that $\langle X_{j} \rangle = 0$ for $1 \leq j \leq M'$ and subsets $S = \{s_{1}, \ldots, s_{M} \}$ and $T= \{t_{1}, t_{2}, \ldots, t_{M} \}$  of $\{1, \ldots, M' \}$ such that  $\langle X_{{s_i}}X_{t_{j}} \rangle = A_{s_{i}t_{j}}$. That is, we have that $A$ appears as a submatrix given by the intersection of rows given by $S$ and columns given by $T$ of the covariance matrix of Gaussian field  $(X_{1}, \ldots, X_{M'})$. We now show that any $M \times M$ matrix $A$ is $2M$-Gaussian embeddable. We refer to the $M' \times M'$ covariance matrix $C$ whose entries are given by $C_{ij} = \langle X_{i} X_{j} \rangle $ as a \textit{Gaussian embedding} of $A$. 
\begin{proposition} \label{prop3.2}
Let $A$ be an arbitrary $M \times M$ real matrix, then $A$ is $2M$-Gaussian embeddable with a Gaussian embedding given by $C = B + \alpha{I}$ where $\alpha \geq \|A\|$ and $B$ is the  $2M \times 2M$ matrix given in block form as
$$
B = \begin{bmatrix}
0&A\\
A^{T}&0
\end{bmatrix}.
$$  \label{4.2}
\end{proposition}
\begin{proof}

We have that $B$ is a symmetric $2M \times 2M$ and therefore has real eigenvalues which we list in decreasing order in terms of absolute value as $|\lambda_{1}(B)| \geq |\lambda_{2}(B)| \geq \cdots \geq |\lambda_{2M}(B)|$. In particular, we have that $\|B\|= |\lambda_{1}(B)|$. For any $\alpha \geq |\lambda_{1}(B)|$, we have that the matrix $C = B + \alpha{I}$ is positive semi-definite and can be associated to a centered Gaussian field $(X_{1}, \ldots, X_{2M})$ where $\langle X_{i}X_{j} \rangle = C_{ij}$ Furthermore, from lemma \ref{eigB} below we have that $|\lambda_{1}(B)|= \|A\|$. From here, we have that $C_{S,T}=A$ for $S=\{1, \ldots, M \}$ and $T = \{M+1, \ldots, 2M \}$ which precisely means that $A$ is Gaussian embeddable.
\end{proof}

For a symmetric matrix $A$, we denote  $$\rho(A) = \max \{ |\lambda_{i}(A)| : \text{ $\lambda_{i}$ is an eigenvalue of $A$} \}$$
and refer to it as the \textit{spectral radius of $A$}. In order to apply our algorithm to an arbitrary real matrix $A$ from the previous section we need to find $\alpha \geq \rho(B)$ for the block matrix $B$ defined in proposition \ref{prop3.2}. The lemma below relates $\rho(B)$ and $\|A\|$.
\begin{lemma} \label{eigB} 
Let $A$ be an $M \times M$ real matrix and let $B$ be the symmetric matrix $$
B = \begin{bmatrix}
0&A\\
A^{T}&0
\end{bmatrix}.
$$
then $\|A\| = \rho(B)=\|B\|$.
\end{lemma}
\begin{proof}
Since $B$ is symmetric, we have that $\|B\|=|\lambda_{1}(B)|=\rho(B)$. We first show that $\|A \| \leq \rho(B)$. Indeed, we have that there exists vectors $v_{1}$ and $v_{2}$ of unit norm such that $\|A \| = \|Av_{1}\|_{2}$ and  $\|A^{T} \| = \|A^{T}v_{2}\|_{2}$. We then set $w = [v_{1}, v_{2}]^{T}$ and observe that
\begin{align*}
    (\rho(B))^{2} &= \|B \|^{2} \\
                  &\geq \frac{\|Bw \|_{2}^{2}}{\|w\|^{2}_{2}} \\
                  &= \frac{1}{2} \big( \|Av_{2}\|^{2}_{2} + \|A^{T}v_{1} \|^{2}_{2} \big) \\
                  &= \frac{1}{2} \big( \|A\|^{2}+ \|A^{T} \|^{2} \big) \\
                  &= \|A\|^{2}.
\end{align*}
To show that $\|A \| \geq \rho(B)$, we have that for any $w = [v_{1}, v_{2}]^{T} \in \mathbb{R}^{2m}$ that 
\begin{align*}
    \|Bw\|^{2}_{2} &= \|Av_{2}\|^{2}_{2} + \|A^{T}v_{1} \|^{2}_{2} \\
                   &\leq \|A\|^{2}\|v_{2}\|^{2}_{2} + \|A^{T}\|^{2}\|v_{1}\|^{2}_{2} \\
                   &= \|A\|^{2}\|w\|^{2}_{2}.
\end{align*}
It follows that $\|Bw\|_{2} \leq \|A\|\|w\|_{2}$ and the claim then follows from the definition of the operator norm and $\rho(B) = \|B\|$.
\end{proof}
 We now present the main theorem needed for estimating the permanent. 
\begin{theorem} \label{X1...X2M}
Let $A$ be an arbitrary $M \times M$ real matrix and let $B$ be the matrix given in block form as
$$
B = \begin{bmatrix}
0&A\\
A^{T}&0
\end{bmatrix}
$$ 
and let $C = B + \alpha{I}$ be positive semidefinite. Now let $(X_{1}, \ldots, X_{2M})$ be the Gaussian field associated to $C$. We have that
\begin{equation}
    \mathrm{perm}(A) = \big\langle \prod^{2M}_{j=1} X_{j} \big\rangle.
\end{equation}
\end{theorem}
\begin{proof}
We first partition the set of random variables  $\{X_{1}, \ldots, X_{M}, X_{M+1}, \ldots, X_{2M} \}$ into the two groups $S_{1} = \{X_{1}, \ldots, X_{M} \}$ and $S_{2} =\{X_{M+1}, \ldots, X_{2M} \}$. 
\newline \indent We can associate each $X_{j}$ with $:X_{j}:$ to see that $$\big\langle \prod^{2M}_{j=1} X_{j} \big\rangle = \big\langle \prod^{2M}_{j=1} :X_{j}: \big\rangle.$$
 By proposition \ref{prop2.3}, the expectation 
 $\big\langle \prod^{2M}_{j=1} :X_{j}: \big\rangle$ is equal to the sum of values $\nu(\gamma)$ complete Feynman diagrams $\gamma$ on the set of Wick products $:X_{i}:$ each of which can be thought of as a vertex $i$ with a single leg. In this case, we have that Feynman diagrams that involve pairings between two vertices within either $S_{1}$ or $S_{2}$ have value zero. It then follows, that the only complete Feynman diagrams whose values contribute to the sum are those which involve pairings between the $M$ vertices in $S_{1}$ to the $M$ vertices in $S_{2}$. In particular, we have that
 $$ \big\langle \prod^{2M}_{j=1} :X_{j}: \big\rangle  = \langle  :\prod_{j=1}^{M}X_{j}::\prod^{2M}_{k=M+1}X_{k}: \rangle.$$ The claim follows since $\langle X_{j_{1}}X_{M+j_{2}} \rangle = A_{j_{1},j_{2}}$ and the right-hand side above is equal to  $\mathrm{perm}(A)$.
\end{proof}

We present the algorithm for an arbitrary real matrix $A \in \mathbb{R}^{M \times M}$ below.

\begin{enumerate}
    \item Given $A$, compute the associated covariance matrix $C \in \mathbb{R}^{2M \times 2M}$ given by $B + \alpha{I}$ for $\alpha \geq \|A\|$. 
    \item For a fixed $N$, we draw $N$ samples $X^{(k)}  = [X^{(k)}_{1}, \ldots, X^{(k)}_{2M}]$ from the multivariate normal  $\mathcal{N}(0, C)$.
    \item For each sample, $X^{(k)}$, compute the product $\mu_{k} = X^{(k)}_{1}\ldots{X^{(k)}_{2M}}$.
    \item We then compute the sample mean as an output $S_{N} = \frac{1}{N} \sum^{N}_{k=1} \mu_{k}.$
\end{enumerate}
\indent
For step 1, it is noted that one can choose any $\alpha \geq \|A\|$ by lemma \ref{eigB}. This is so that the matrix $C = B + \alpha{I}$ is positive semidefinite. 
\newline \indent 
 To analyze the total time complexity, we assume the operations of addition and multiplication over the field of real numbers take unit time. We also assume that one can sample from the standard normal $\mathcal{N}(0, I)$ in unit time. To set $\alpha = \|A\|$, one can for example diagonalize $B$ which takes time $O(M^{3})$. To obtain $N$ samples from $\mathcal{N}(0,B+\alpha{I})$, we need to compute a Cholesky factorization $L$ of $B+{\alpha}I$ which takes $O(M^{3})$ time and then compute the product $LX$, where $X \sim \mathcal{N}(0,I)$, $N$ times resulting in an additional $O(M^{2}N)$ computation time.  This brings the cost of sampling $N$ samples to be $O(M^{3}+M^{2}N)$ in the worst case without approximation. After sampling, however, the procedure scales linearly in the number of samples as each term $\mu_{k}$ can be computed in time $O(M)$ bringing the computation after sampling to be $O(MN)$ and total computation to be $O(M^{3}+M^{2}N + MN)$.
\newline \indent 
In dealing with classes of matrices where the operator norm is known to be bounded by $d$, we can choose $\alpha = d$. In general situations, we can also take $\alpha$ the Frobenius norm $\|A \|_{F}$, which we recall is given by $\|A \|_{F} = \sqrt{\sum_{i,j} |A_{ij}|^{2}}$ for an arbitrary $M \times M$ real matrix $A$. In the first step, it suffices to take $\alpha > \|A\|_{F}$. This follows from a standard result that the operator norm of a matrix is bounded above by its Frobenius norm and include a proof below for completeness.
\begin{lemma}
Let $A$ be an arbitrary $M \times M$ real matrix then $ \| A\| \leq \|A \|_{F}$.
\end{lemma}
\begin{proof}
Let $v = \sum^{M}_{i=1} v_{i}e_{i}$ for real coefficients $v_{i} \in \mathbb{R}$ and standard basis vectors $e_{i}$. We then have from the Cauchy-Schwarz inequality that
$$ \|Av \|^{2}_{2} \leq \big(\sum |v_{i}|^{2} \big) \big(\sum \|Ae_{i}\|_{2}^{2} \big).$$
The claim then follows from the definitions of the operator and Frobenius norms.
\end{proof}

 We now provide an error analysis for the algorithm.

\begin{proposition} \label{failureprob}
Let $S_{N}$ be the sample mean defined in the above algorithm for estimating the permanent of a real $M \times M$ matrix $A$ where we take $\alpha \geq 0$ to be any real number such that $B + \alpha{I}$ is positive definite. Then for fixed $N$, we have that  
\begin{equation}
    P(|S_{N}-\text{perm}(A)| > t) \leq \frac{3^{M}}{t^{2}N}\alpha^{2M}.
\end{equation}
\end{proposition}

\begin{proof}

We have the following by Chebyshev's inequality and independence
\begin{align*}
   P(|S_{N}-\text{perm}(A)| > t) &\leq  \frac{1}{t^{2}}Var(S_{N}) \\
   &= \frac{1}{N^{2}t^{2}}\sum^{N}_{k=1}Var(\mu_{k}).
\end{align*}
We now need to bound the variance $Var(\mu_{k}) = Var( \prod^{2M}_{i=1} X_{i})$. For this, we can apply a second moment bound and by  the Cauchy-Schwarz inequality,
$$Var( \prod^{2M}_{i=1} X_{i}) \leq \langle (\prod^{2M}_{i=1} X_{i})^{2} \rangle  \leq   \langle (\prod^{M}_{i=1} X_{i})^{4} \rangle^{1/2}  \langle (\prod^{2M}_{i=M+1} X_{i})^{4} \rangle^{1/2}.$$
Now since $(X_{i})^{M}_{i=1}$ are independent and $(X_{i})^{2M}_{i=M+1}$ are independent, we obtain that  
\begin{align*}
     \big\langle (\prod^{M}_{i=1} X_{i})^{4} \big\rangle^{1/2}  \big\langle (\prod^{2M}_{i=M+1} X_{i})^{4} \big\rangle^{1/2} &=  (\prod^{M}_{i=1} \big\langle X_{i}^{4} \big\rangle)^{1/2}  (\prod^{2M}_{i=M+1} \big\langle X_{i}^{4} \big\rangle)^{1/2}.
\end{align*}

Since $X_{i} \sim \mathcal{N}(0, \alpha)$, we have that $\langle X_{i}^{4} \rangle = 3\alpha^{2}$ hence the claim follows.

\end{proof}

\begin{remark} \label{remark 1}
For general products of multivariate Gaussians, one can use Nelsen's hypercontractivity to obtain bounds on the $p$-norm of products for $p \geq 2$. See Theorem 1.22 in \cite{S}. Using this theorem and a fourth moment method, a similar argument to proposition \ref{failureprob} yields 
 $$P(|S_{N}-\text{perm}(A)| > t) \leq \frac{1}{t^{4}} \bigg(\frac{3^{M}}{N^{3}}+ \frac{3}{N^{2}} \bigg)\bigg(3^{M}\alpha^{2M} \bigg)^{2}.$$
but we do not focus on this estimate in this work.
\end{remark}
In particular, if we take $\alpha = \|A\|$ and $N$ to be $O(\frac{1}{\epsilon^{2}})$ then we obtain an estimate of $\mathrm{perm}(A)$ within an additive error of $\epsilon (\sqrt{3}\|A\|)^{M}$ in polynomial time. We can compare this to algorithm of Gurvits, which obtains an estimate of within an additive error of $\epsilon\|A\|^{M}$. It should be noted that Gurvits showed that $\text{perm}(A) \leq \|A\|^{M}$ which means our estimate is a potentially weak one, but we can improve it by modifying the covariance matrix used in our procedure. 

\section{Adjusting the Covariance}
\subsection{Additive FPRAS}
Often in approximation theory, the goal is to obtain a fully polynomial randomized approximation scheme (FPRAS) for the quantity that is being approximated.  In our case of estimating the permanent up to an additive error, this is a randomized algorithm, which given as an input an $M \times M$ matrix $A$ together with a fixed accuracy parameter $t \in (0, 1]$, outputs a number $Z$ in time polynomial in $M$ and $t^{-1}$ such that
    $$P(|Z-\text{perm}(A)| > t) \leq \frac{1}{4}.$$
 This formalization takes into account the number of samples needed to obtain the desired failure probability for fixed $t$. 
    \newline \indent
In the case of our procedure, we obtained as a result of proposition \ref{failureprob} for fixed $t$ that for an $M \times M$ real matrix $A$ that $$P(|S_{N}-\text{perm}(A)| > t) \leq \frac{3^{M}}{t^{2}N}\alpha^{2M}$$
where $N$ is the number of samples taken for $S_{N}$ which takes $O(MN)$ time to compute after obtaining $N$ samples from the distribution $\mathcal{N}(0, B+\alpha{I})$. If we take $\alpha = \|A\|$ and assume that $\sqrt{3}\|A\| < 1$ then taking $N = \frac{4}{t^{2}}$ samples allows us to compute $S_{N}$ in time polynomial in $M$ and $t^{-1}$ such that 
$$P(|S_{N}-\text{perm}(A)| > t) \leq \frac{1}{4}.$$
In particular,  if $\sqrt{3}\|A\| < 1$ we can obtain an FPRAS for $\text{perm}(A)$ in the above sense.
   % \newline \indent
%However, it is important to note that Gurvits in \cite{G} showed that $|\text{perm}(A)| \leq \|A\|^{M}$ (see also pg. 7 of \cite{AH}). This means that an additive estimate for an arbitrary fixed $t$ may not be meaningful. For the case of our procedure, we have for any $c>0$ that if one sets $t = c(\sqrt{3}\alpha)^{M}$ for $\alpha \geq \|A\|$ then one obtains that $P(|S_{N}-\text{perm}(A)| > t) \leq {1}/{c^{2}N}.$ 
\newline \indent
As mentioned above that our failure probability in proposition \ref{failureprob} is bounded in terms of $\alpha \geq \|A\|$. We chose such an $\alpha$ so that $B + \alpha{I}$ is positive semi-definite. By Schur's determinant formula the eigenvalues of $B$ are $$\pm \lambda_{1}(A^{T}A)^{1/2}, \ldots, \pm \lambda_{M}(A^{T}A)^{1/2}.$$  In particular by lemma \ref{eigB},  $-\|A\| \leq 0$ is the smallest eigenvalue of $B$. Therefore, $B + \alpha{I}$ is positive semidefinite if and only if $\alpha \geq \|A \|$. Therefore, we cannot improve our choice of $\alpha$.

%Unfortunately, we can not improve the choice of $\alpha$ since Schur's determinant formula implies that eigenvalues of $B$ %are $\pm \lambda_{1}(A^{T}A)^{1/2}, \ldots, \pm \lambda_{M}(A^{T}A)^{1/2}$. In particular by lemma 3.1, since $\|A\|$  is an eigenvalue of $B$ then $-\|A\|$ is also an eigenvalue of $B$.

We can however modify our covariance matrix by replacing the matrix $\alpha{I}$ with any diagonal matrix $D$ such that $C = B + D \succeq 0$. Here we use the notation $C \succeq 0$ to denote that the matrix $C$ is positive semi-definite.  We have the following similar result to theorem \ref{X1...X2M} and which is proven by the same argument.

    \begin{theorem} \label{genpermexp}
    Let $A$ be an arbitrary $M \times M$ real matrix and let $B$ be the matrix given in block form as
    $$
    B = \begin{bmatrix}
    0&A\\
    A^{T}&0
    \end{bmatrix}.
    $$ 
    Furthermore, let $D$ be a diagonal matrix such that $B + D \succeq 0$. Now let $(X_{1}, \ldots, X_{2M})$ be the Gaussian field associated to $B+D$. We have that
    \begin{equation*}
        \mathrm{perm}(A) = \big\langle \prod^{2M}_{j=1} X_{j} \big\rangle.
    \end{equation*}
    \end{theorem}
    
    We then can adjust our algorithm by sampling from $\mathcal{N}(0, B+D)$ and obtain the following error estimate by the same argument as proposition \ref{failureprob}.
    
    \begin{proposition} \label{genfailureprob}
                    Let $S_{N}$ be the sample mean defined in the adjusted algorithm with samples from $\mathcal{N}(0, B+D)$.  Then for fixed $N$, we have that  
    \begin{equation}
        P(|S_{N}-\text{perm}(A)| > t) \leq \frac{3^{M}}{t^{2}N} \prod^{2M}_{i=1} D_{ii}.
    \end{equation}
    \end{proposition}

    In particular, we have that sampling with covariance $B+D$ forms an additive FPRAS when $\prod^{2M}_{i=1} D_{ii} < 3^{-M}$ .  As before, we can estimate $\text{perm}(A)$  to an additive error $\epsilon \bigg(\sqrt{3\prod^{2M}_{i=1} D_{ii}}\bigg)^{M}$ in polynomial time by taking $O(\frac{1}{\epsilon^{2}})$ samples. Furthermore, this algorithm improves over algorithm in \cite{G} when $3^{M} \prod^{2M}_{i=1} D_{ii} \leq \|A\|^{2M}$. When $D = \alpha{I}$ for $\alpha \geq \|A\|$, we have that $\prod^{2M}_{i=1} D_{ii} = \alpha^{2M}$, therefore this estimate agrees with the estimate given by proposition 3.3. We now discuss how to find $D+B$ in the subsequent subsections. 

    \begin{remark}
    As in remark \ref{remark 1}, a fourth moment method and hypercontractivity yields
     $$P(|S_{N}-\text{perm}(A)| > t) \leq \frac{1}{t^{4}} \bigg(\frac{3^{M}}{N^{3}}+ \frac{3}{N^{2}} \bigg)\bigg(3^{M} \prod^{2M}_{i=1} D_{ii} \bigg)^{2}$$ but we do not focus on this estimate in this work. 
    \end{remark}
     \subsection{Diagonally-dominant matrices}
     One class of matrices such that $B + D \succeq 0$ is provided when $D$ is such that $B + D$ is diagonally dominant. Recall that a symmetric matrix $L$ is said to be diagonally dominant if the diagonal entry of each row is greater than the sum of the absolute values of the other entries of each row, that is $L_{ii} \geq \sum_{j \neq i} |L_{ij}|$.   Diagonally dominant matrices are positive semi-definite by the Gershgorin circle theorem. Given a square matrix $B$, we define the diagonal matrix $D_{d}(B)$ by the diagonal entries $(D_{d}(B))_{ii} = \sum_{j} |B_{ij}|$ for all $i$. We often refer to $D_{d}(B)$ as $D_{d}$ when the context is clear. 
     \newline \indent
    An example of diagonally dominant matrices are Laplacian matrices which are matrices $L$ such that $L_{ij} \leq 0$  when $i \neq j$ and $\sum_{j} L_{ij} = 0$ for all $i$. Laplacian matrices can be viewed as the combinatorial Laplacian of non-negative weighted graphs. For any weighted graph $G = (V,E,w)$ with weighted adjacency matrix given by a symmetric matrix $w$ with $w_{ii} = 0$ for all $i$ and $w_{ij} \geq 0$ for all $(i,j)$, we define the combinatorial Laplacian of $G$ to be the Laplacian matrix $C$ given by $C  = D_{d}-w$. 
        \newline \indent
     Given $B$ as before, one can compute the diagonal matrix $D_{d}$ with diagonal entries $(D_{d})_{ii} = \sum_{j \neq i} |B_{ij}|$ in worst case $O(M^{2})$. In the case when $A$ has non-negative entries, we have that the centered Gaussian field with covariance matrix given by $B+D_{d}$ provides a better bound on the failure probability of proposition \ref{genfailureprob} then $B + \alpha{I}$ for $\alpha \geq \|A\|$ as follows from proposition \ref{Diagdom} below. This proposition is a consequence of the arithmetic-geometric mean inequality. 
     \begin{proposition} \label{Diagdom}
     Let $A$ be a real $M \times M$ matrix with non-negative entries and let $B$ be the block matrix $$
    B = \begin{bmatrix}
    0&A\\
    A^{T}&0
    \end{bmatrix}.
    $$ 
    Furthermore, let $D_{d}$ be the diagonal matrix with diagonal entries $(D_{d})_{ii} = \sum_{j \neq 0} |B_{ij}|$. We then have that $$ \prod^{2M}_{i=1} (D_{d})_{ii} \leq \alpha^{2M}$$
    for $\alpha \geq \|A\|$.
     \end{proposition}
    \begin{proof}
    Let $\textbf{1}$ be the vector of ones. We then have by definition that 
    \begin{align*}
        \|A\| &= \sup_{|v|=1} v^{T}Av \\
            &\geq \frac{1}{M} \textbf{1}^{T}A \textbf{1} \\
            &= \frac{1}{M} \sum^{M}_{k=1} A_{k}    
    \end{align*}
    where $A_{k} = \sum_{j=1}^{M} A_{kj}$ is the $kth$ row sum of $A$. Then by the arithmetic-geometric mean inequality, we have that 
    $$\frac{1}{M} \sum^{M}_{k=1} A_{k} \geq (\prod^{M}_{k=1} A_{k})^{1/M}.$$
    It follows then that $\|A\|^{M} \geq \prod^{M}_{k=1} A_{k} $. In particular, for $$B = \begin{bmatrix}
        0 & A \\
        A^{T}&0 ,
    \end{bmatrix}$$
    we have that for $\alpha \geq \|A \|$, that 
    $$ \alpha^{2M} \geq \prod_{k=1}^{M} A_{k} \prod^{M}_{k=1}A^{T}_{k}.$$
    \end{proof}
     The statement of proposition 3.5 does not hold when the entries of $A$ are possibly negative. For example, if we set $A$ to be the matrix $\begin{bmatrix}
         1&1\\ 1&-1
     \end{bmatrix}$
     we have that $\|A\| = \sqrt{2}$ but $D_{d}$ is the diagonal matrix with diagonal entries equal to $2$.  
     \newline \indent 
      To give an example of a case where sampling from the covariance matrix $D_{d} + B$ improves over the case of sampling from the covariance $B + \|A\|{I}$, consider the case where $A$ is taken to be the matrix 
    $$A = \frac{1}{6}\begin{bmatrix}
        0&1&1&1&1&1&1&1&1&1\\
        1&1&1&0&0&0&0&0&0&0 \\
        1&1&0&1&0&0&0&0&0&0 \\
        1&0&1&0&1&0&0&0&0&0 \\
        1&0&0&1&0&1&0&0&0&0 \\
        1&0&0&0&1&0&1&0&0&0 \\
        1&0&0&0&0&1&0&1&0&0 \\
        1&0&0&0&0&0&1&0&1&0 \\
        1&0&0&0&0&0&0&1&0&1 \\
        1&0&0&0&0&0&0&0&1&1 \\
    \end{bmatrix}.$$
    We then have sampling from $B+ D_{d}$ forms an FPRAS for approximating the permanent of $A$. If we instead sampled from $B + \alpha{I}$ for $\alpha \geq \|A\|$ to approximate the permanent then since $\sqrt{3}\|A \| > 1$, we would not obtain an FPRAS for approximating $\text{perm}(A)$. By proposition \ref{genfailureprob}, we have that sampling from $B + D_{d}$ forms an FPRAS for estimating the permanent of $A$, since in this case $\det(D_{d}) < 3^{-10}$ which implies that $3^{10}\det(D_{d}) < 1$.
    \newline \indent We now highlight an application of our results of diagonally dominant matrices to counting perfect matchings of bipartite graphs. Observe that given a bipartite graph $G = (V_{1}, V_{2}, E)$ with $|V_{1}| = |V_{2}|= M$, the  adjacency matrix $B$ of $G$ has the block form 
    $$B  = \begin{bmatrix}
        0&A \\
        A^{T}&0
    \end{bmatrix}$$
    and the number of perfect matchings of $G$ is given by the $\text{perm}(A)$. By proposition \ref{genfailureprob}, if we take our covariance matrix to be $C = D_{d} + B$ then by taking $K/\epsilon^{2}$ samples,  our algorithm estimates $\text{perm}(A)$ within an  additive error of $\epsilon\bigg(\sqrt{3^{2M}\prod^{2M}_{i=1}({D}_{d})_{ii}}\bigg)$ with probability at least $1- K^{-1}.$ But since $(D_{d})_{ii}$ is the degree $\deg(v_{i})$ of the $i$th vertex $v_{i}$ of $G$, we have that our algorithm estimates the number of perfect matchings of $G$ within an additive error $\epsilon\bigg(\sqrt{3^{2M}\prod^{2M}_{i=1}deg(v_{i})}\bigg)$. If we wish to find $G$ where the upper bound of the failure probability of proposition \ref{genfailureprob} improves on the procedure in \cite{G}, we would need the degrees of $G$ to satisfy   $$3^{M} \prod_{i=1}^{2M} deg(v_{i}) \leq \lambda_{1}(B)^{2M}$$
    where we used lemma \ref{eigB} for the right hand side. By the arithmetic-geometric mean inequality, it suffices to find $G$ where the vertices satisfy
    $$ \frac{1}{2M} \sum_{i=1}^{2M} deg(v_{i}) \leq \frac{\lambda_{1}(B)}{\sqrt{3}}.$$
    This occurs for example, if $G$ is the bipartite double cover of a star graph on $M$ vertices for $M \geq 20$. It is likely many examples exist, for instance, one can view $\frac{1}{2M} \sum_{i=1}^{2M} deg(v_{i}) = d_{avg}$ where $d_{avg}$ is the average degree of $G$ then if we let $d_{max} = \max \{ deg(v_{i})  \}$ then it is known that $\lambda_{1}(B) \in [d_{avg}, d_{max}]$. 

    \subsection{A semidefinite program to reduce variance}
   We begin this subsection by recalling primal and dual problems of semidefinite programming. For symmetric matrices $S$ and $S'$, we use the notation $S \succeq S'$ to mean that the matrix $S-S'$ is postive semidefinite and use $\langle S, S' \rangle_{F} = tr(SS')$ to denote the Frobenius inner product.   Let $A_{1}$, \ldots $A_{m}, C$  be symmetric real $n \times n$ matrices and $b_{1}, \ldots, b_{m}$ to be real scalars.  A semidefinite program with $n^{2}$ variables (regarded as an $n \times n$ matrix) and $m$ constraints can be written as solving the following problem over symmetric matrices $X$,
    \begin{equation} \label{primal}
\begin{array}{ll@{}ll}
    \text{maximize} & \langle C, X \rangle_{F} \\
    \text{subject to}& \langle X, A_{j} \rangle_{F} \leq b_{j},  &&1 \leq j \leq m. \\
    &X \succeq 0 
\end{array}.
\end{equation}
    We refer to a matrix $X$ that satisfies the constraints given by \ref{primal} as a feasible candidate. We follow the convention to refer to the problem \ref{primal} as the \textit{primal} problem. We write the corresponding \textit{dual} problem as
        \begin{equation} \label{dual}
\begin{array}{ll@{}ll}
    \text{minimize} & b^{T}y \\
    \text{subject to}&  \sum^{m}_{j=1} y_{j}A_{j} -C \succeq 0 \\
    &y_{j} \geq 0 &&1 \leq j \leq m.
\end{array}.
\end{equation}
where we took $b = [b_{1}, \ldots, b_{m}]^{T}$. We refer to a vector $y \geq 0$ that satisfies the constraints of \ref{dual} as a dual feasible candidate. It well known that the weak duality holds for semidefinite programs,  that if $y^{*}$ is a solution of \ref{dual} and $X$ is solution of \ref{primal} then $b^{T}y^{*} \geq \langle C,X \rangle_{F}$. When $b^{T}y^{*} = \langle C,X \rangle_{F}$, we say strong duality holds.
    \newline \indent We can consider now optimizing the right hand side of proposition \ref{genfailureprob} over all $D$ such that $B+D \succeq 0$. We can state this formally as the following optimization program
    \begin{equation}
    \begin{array}{ll@{}ll}
    \text{minimize}  & \det(D) &\\
    \text{subject to}&  B+ D \succeq 0 \\
    \end{array}
    \tag{$P_{1}$}.
    \end{equation}
    We are unaware of convex interpretation of the above program for which a polynomial time procedure can be applied, but we can relax the problem by considering the trace as the objective function. This yields the following 
    \begin{equation*}
    \begin{array}{ll@{}ll}
    \text{minimize}  &  trace(D) &\\
    \text{subject to}&  B+ D \succeq 0 \\
    \end{array}
    \end{equation*}
    which can be rephrased as the dual semi-definite program
    \begin{equation} \label{OurSDP}
    \begin{array}{ll@{}ll}
    \text{minimize}  &  \textbf{1}^{T}x &\\
    \text{subject to}&  \sum x_{i}e_{i}e_{i}^{T} + B \succeq 0 .
    \end{array}
    \tag{$P_{2}$}.
    \end{equation}
    where we take $\textbf{1} \in \mathbb{R}^{2M}$ to be the vector consisting of ones and $x \in \mathbb{R}^{2M}$ to be the vector formed by the diagonal entries of $D$. 
    This program has the benefit of being well-studied and admitting a number of polynomial time numerical algorithms that obtain an approximate $\hat{x}$ to a solution $x^{*}$ of $P_{2}$. 
    \newline \indent 
    In reference to proposition \ref{genfailureprob}, if we let $D^{*}$ be the diagonal matrix with diagonal entries $D^{*}_{ii} = x^{*}_{i}$ then we have that sampling from the covariance $D^{*} + B$ improves over the case where we are sampling from $\alpha{I} + B$ for $\alpha \geq \|A\|$. This is stated formally in the proposition below that is again proved by the use of the arithmetic-geometric mean inequality.
    \begin{proposition}
    Let $A$ be a real $M \times M$ matrix and let $B$ be the block matrix $$
    B = \begin{bmatrix}
    0&A\\
    A^{T}&0
    \end{bmatrix}.
    $$ 
    Futhermore, let $D^{*}$ be diagonal matrix corresponding to the solution in $P_{2}$. We have that
    $$ \prod^{2M}_{i=1} D^{*}_{ii} \leq \alpha^{2M}$$
    for $\alpha \geq \|A\|$.
    \end{proposition}
    \begin{proof}
    Since $B + \alpha{I} \succeq 0$, we have that $trace(D^{*}) \leq trace(\alpha{I})$, in particular, we have that ${\sum_{i} D^{*}_{ii}}/2M \leq \alpha$. Therefore we have that $$\det(D^{*})^{1/2M} = (\prod_{i=1}^{2M} D^{*}_{ii})^{1/2M} \leq \sum_{i}  D^{*}_{ii}/2M  \leq \alpha = \det(\alpha{I})^{1/2M}.$$
    \end{proof}
    %We then have that\prod^{2M}_{i=1} \hat{D}_{ii} \leq \prod^{2M}_{i=1} (D_{d})_{ii} $$ with equality if and only if (CHECK THE OTHER DIRECTION) when the entries of $A$ are non-negative
    If we translate the convex set $\{x : \sum x_{i}e_{i}e_{i}^{T} + B \succeq 0 \}$ in $P_{2}$ by the diagonal matrix $D_{d}$ defined in proposition \ref{Diagdom}, we obtain the following semi-definite program
    
    \begin{equation} \label{MaxCut}
    \begin{array}{ll@{}ll}
    \text{minimize}  &  \textbf{1}^{T}x &\\
    \text{subject to}&  \sum x_{i}e_{i}e_{i}^{T}  \succeq D_{d} - B .
    \end{array}
    \tag{$P_{3}$}
    \end{equation}

    The solutions to problems \ref{OurSDP} and \ref{MaxCut} are related, in particular, if  $y^{*}$ is a solution to \ref{MaxCut} if and only if $\sum y^{*}_{i}e_{i}e_{i}^{T} - D_{d}$ is be a solution of \ref{OurSDP}.
    \newline \indent
  The problem \ref{MaxCut} is the dual problem of a specific example of well-known semi-definite program that was used in work of Goemans-Williamson \cite{GW} which we now briefly recall.   %Mention dual problem
    \newline \indent 
   Following \cite{GW}, let $G = (V,E)$ be an arbitrary graph with non-negative edge weights $w_{ij} \geq 0$. We define a cut to a disjoint partition $(S, \bar{S})$ of the vertex set $V$ and the weight of a cut, $w(S,\bar{S}) = \sum_{i \in S, j \in \bar{S}} w_{ij}$. The Max-Cut problem seeks to find a cut $(S,\bar{S})$ that maximizes $w(S,\bar{S})$. This problem is known to be $NP$-complete \cite{K} and can be formulated as the following optimization problem.
    \begin{equation}
    \begin{array}{ll@{}ll}
    \text{max}  &  \frac{1}{2}\sum_{(i,j) \in E} w_{ij}(1-y_{i}y_{j}) &\\
    \text{where}& y_{i} \in \{-1,1\}
    \end{array}.
    \label{IntProgram}
    \end{equation}
    In \cite{GW}, Goemans-Williamson considered a semidefinite program that was shown by Arora and Kale in \cite{AK} to be equivalent to the following primal problem.
    \begin{equation} \label{MaxCutprimal}
    %\begin{array}{ll@{}ll}
    %\text{max}  &  \sum_{(i,j) \in E} w_{ij}\|v_{i}-v_{j} \|^{2} &\\
    %\text{where}& \|v_{i}\|^{2} \leq 1 \text{ for all $i \in V$}.
    %\end{array}
    \begin{array}{ll@{}ll}
    \text{maximize} & \langle C, X \rangle_{F} \\
    \text{subject to}& X_{ii} \leq 1,  &&i \in V \\
    &X \succeq 0 
\end{array}
    \end{equation}

    The Goemans-Williamson approximation to the Max-Cut problem \eqref{IntProgram} proceeds by first obtaining a solution $X$  to \eqref{MaxCutprimal} with Gram vectors $v_{1}, \ldots, v_{n}$. The algorithm then produces a random cut $(S,\bar{S})$ by drawing a unit vector $u$ uniformly from the unit sphere in $\mathbb{R}^{n}$ and setting $S = \{ i: \langle v_{i}, u \rangle \geq 0 \}$ and $\bar{S} = \{i: \langle v_{i}, u \rangle < 0 \}$. The  weight $w(S,\bar{S})$ has expectation that was shown in \cite{GW} to satisfy $$E[w(S,\bar{S})] \geq 0.878MAXCUT(G)$$ where $MAXCUT(G)$ is the value of the Max-Cut of $G$. 
    \newline \indent
    We have that the block matrix $B$ can be thought of as the adjacency matrix of the real weighted bipartite graph $G = (V_{1}, V_{2}, E)$ where $B_{ij}$ indicates the weight of the edge $(i,j)$. In particular, when $B_{ij} \geq 0$ we have that $C = D_{d}-B$ is the combinatorial Laplacian matrix of the bipartite graph $G$ and \ref{MaxCut} is the dual problem to the primal problem \eqref{MaxCutprimal}. The primal problem \eqref{MaxCutprimal} can be thought of as the semidefinite program associated to the Max-Cut problem \eqref{IntProgram} of the bipartite graph $G$. For this case, since $B_{ij} \geq 0$, then $MAXCUT(G)$ is attained when the cut $(S, \bar{S})$ is $(V_{1}, V_{2})$. This is the idea behind the following proposition that says the solution to \ref{OurSDP} is attained by the diagonally dominant case when $B_{ij} \geq 0$. 

    \begin{proposition} \label{Prop4.4}
    Let $A$ be a real $M \times M$ matrix with non-negative entries and let $B$ be the block matrix $$
    B = \begin{bmatrix}
    0&A\\
    A^{T}&0 \end{bmatrix}$$ then the solution to \ref{OurSDP} is obtained at $x = diag(D_{d})$.
    \end{proposition}
    \begin{proof}
    It suffices to show that a solution to \ref{MaxCut} is attained at $2diag(D_{d})$. For this problem, we consider the primal problem which is the problem dual to \ref{MaxCut} stated as
\begin{equation}
\begin{array}{ll@{}ll}
    \text{maximize} & \langle C, X \rangle_{F} \\
    \text{subject to}& X_{ii} \leq 1,  &&i \in V \\
    &X \succeq 0 
\end{array}
\end{equation}
where any primal candidate $X$ is assumed to be symmetric, the inner product $\langle \cdot, \cdot \rangle_{F}$ is the Frobenius inner product, $V$ is the set $\{1, \ldots, 2M \}$, and $C$ is the combinatorial Laplacian of $B$. By Slater's condition \cite{BV}, the strong duality holds, therefore the value of the primal problem is equivalent to the value of the dual problem \ref{MaxCut}. We recall that $X \succeq 0$ if and only if there exists Gram vectors $v_{1}, \ldots, v_{n}$ such that $X_{ij} = v_{i}^{T}v_{j}$. Also, since $C$ is a Laplacian matrix, we can write $\langle C, X \rangle_{F} = \frac{1}{2}\sum_{i,j } B_{ij}\|v_{i}-v_{j} \|^{2}$ for the Gram vectors $v_{1}, \ldots, v_{2M}$ of $X$. Therefore, we can rewrite the primal as finding $v_{1}, \dots, v_{2M}$ such that 
\begin{equation} 
    \begin{array}{ll@{}ll}
    \text{max}  &  \frac{1}{2}\sum_{i,j } B_{ij}\|v_{i}-v_{j} \|^{2} &\\
    \text{where}& \|v_{i}\|^{2} \leq 1 \text{ for all $i \in V$}.
    \end{array}.
\end{equation}  We then have that since $B$ has a bipartite structure and the triangle inequality that $\frac{1}{2}\sum_{i,j } B_{ij}\|v_{i}-v_{j} \|^{2} \leq 4\sum_{j=M+1}^{2M} \sum^{M}_{i=1} B_{ij}$. 
Futhermore, we have that for a  unit vector $u$, the upperbound is attained at $v_{i} = u$ for $i =1, \ldots, M$ and $v_{i}= -u$ for $i=M+1, \ldots, 2M$. Since $4\sum_{j=M+1}^{2M} \sum^{M}_{i=1} B_{ij} = 2trace(D_{d})$, we have proven our claim.
    \end{proof}
For a general real matrix $A$, the matrix $B$ may have negative entries. In this general case, the conclusion of proposition \ref{Prop4.4} may not hold. Therefore one would need to solve the dual program \ref{OurSDP}, the solution of which can be approximated numerically in polynomial time. 
    \newline \indent 
    One approach to \ref{OurSDP} are interior point methods, the first use of which appeared in Karmarker's work \cite{K} in the context of linear programming problems. These methods involve a choice of a barrier function defined on the interior of the feasible region and goes to infinity at the boundary of the region. For our problem \ref{OurSDP}, our feasible set is the set  $K = \{x \in \mathbb{R}^{2M} : S(x) \succeq 0 \}$ where $S(x) = \sum x_{i}e_{i}e_{i}^{T} + B$ as the feasible region. This set is a closed convex subset of the cone of semi-positive definite matrices and interior point methods for such problems were analyzed in \cite{NN1} whose approach we now summarize.
    \newline \indent
    For $K$, we can consider the log-barrier function $\varphi(x) = -\log\det(S(x))$ for $K$. This function is a variation of the log barrier function used in \cite{NN1}. From here, it is clear that $\varphi(x) \to \infty$ as $x \to \partial{K}$. Furthermore, it can shown (see for example \cite{Ne}) that $\varphi$ is $2M$-self-concordant and strictly convex in the interior $int(K)$ of $K$. 
    \newline \indent 
    We then consider the modified objective $\phi_{t}(x) = t\textbf{1}^{T}x + \varphi(x)$, the points $\bar{x}_{t} = \text{argmin}_{x \in S\cap K} \phi_{t}(x)$ form a path in the interior $int(K)$ which tends to a solution $x^{*}$ of $P_{2}$ as $t \to \infty$. This path is referred to as the \textit{central path}. 
    Interior point methods attempt to follow this central path by first finding a initial point and time $(x_{0}, t_{0})$ such that the Newton decrement defined as  $$\lambda_{\phi_{t}}(x) = \sqrt{\nabla\phi_{t}(x)(\nabla^{2}\varphi(x))^{-1}\nabla\phi_{t}(x)}$$ is sufficiently small. 
    The method then proceeds by setting $\hat{x}_{0} = x_{0}$ and increases the time $t_{k} = (1+h)t_{k-1}$ and updates position from $\hat{x}_{k-1}$ to $\hat{x}_{k}$ according to a damped Newton update applied to $\phi_{t_{k}}$. It can then be shown that after $O(\sqrt{M}\log(M/\epsilon))$ iterations that the $k$th iterate $\hat{x}_{k}$ satisfies $\textbf{1}^{T}(\hat{x}_{k}-x^{*}) < \epsilon$ for a solution $x^{*}$ of \ref{OurSDP}. 
    \newline \indent
    The computation of the overall computation cost of the procedure requires the factoring in the cost of each iterate which as mentioned is an application of a damped Newton's method to the function $\phi_{t_{k}}$ at the point $\hat{x}_{k}$. The cost of each iterate is dominated by the cost of computing and inverting the Hessian $\nabla^{2}\phi_{t_{k}}(\hat{x}_{k})$ which takes time $O(M^{4})$. This brings the total cost of the entire procedure to find an approximate to a solution $x^{*}$ of \ref{OurSDP} to ${O}^{*}(M^{4.5})$\footnote{Following \cite{JKL+}, the notation $O^{*}$ hides $M^{o(1)}$ and  $\log^{O(1)}(M/\epsilon)$ factors}. 
    \newline \indent 
    A more recent approach to interior point methods for semidefinite programming in \cite{JKL+} modifies the procedure and can be applied to our problem an approximate solution in time ${O}^{*}(M^{3.5})$. The approach in \cite{JKL+} is for general semi-definite programs. %

\subsection{Relationship to Max-Cut and cut-norms}
In the previous subsection, we discussed how to find a diagonal matrix $D$, through the use of the semidefinite program \ref{OurSDP}. By proposition \ref{genfailureprob}, the upper bound of the failure probability of our algorithm for estimating $\mathrm{perm}(A)$ is given in terms of the diagonal entries of $D$. In particular, we saw in proposition \ref{Prop4.4}, that if $A_{ij} \geq 0$, then the solution to \ref{OurSDP} is attained at $D_{d}$ the trace of which has value $MAXCUT(G)$ for the bipartite graph $G = (V_{1}, V_{2}, E)$ with edge weights given by the weighted adjacency matrix $B$. We now discuss a hint of a connection of the case when the entries of $A$ are real and potentially negative to the problem of finding the Max-Cut of the real weighted bipartite graph $G = (V_{1}, V_{2},E)$. Before doing so, we introduce two matrix norms featured in the work of Alon and Naor in \cite{AN}. 

%Add a description of integer program.

 Let $A$ be an $M_{1} \times M_{2}$ real matrix, we define the cut norm $\|A\|_{C}$ of $A$ to be 
$$\|A \|_{C} = \max \{ |\sum_{i,j} A_{ij}| : I \subset R, J \subset C \} $$ where we took $R = \{1, \ldots, M_{1} \}$ and $C = \{ 1, \ldots, M_{2} \}$. We also define the infinity-one norm 
$$\|A\|_{\infty \to  1 } = \max \{ \sum_{i,j} A_{i,j}x_{i}y_{j} : x_{i} \in \{-1,1 \}, y_{i} \in \{-1,1\} \}.$$
The following relationship between these two norms was given in lemma 3.1 in \cite{AN} which we now state.

\begin{lemma}
\label{cutnormandinf}
Let $A$ be an $M_{1} \times M_{2}$ matrix, then 
$$\|A\|_{C}  \leq  \|A\|_{\infty \to 1 } \leq 4 \|A\|_{C}.$$ Moreover, if the sum of each row  and  the sum of each column of $A$ sum is zero then $\|A\|_{\infty \to 1 }  = 4\|A\|_{C}$.
\end{lemma}
We recall the notation in the previous subsection. Let $G = (V,E)$ be an edge weighted graph with real edge weights $w_{ij}$. We refer to a partition $(S,\bar{S})$ of $V$ as a cut and now let $E(S,\bar{S}) = \{ (i,j) : i \in S, j\in \bar{S} \} \subset E$. We then let $w(S,\bar{S}) = \sum_{(i,j) \in E(S,\bar{S)}} w_{ij}$ be the sum of the weights of the cut and let $\bar{w}(S, \bar{S}) = \sum_{(i,j) \in E \setminus E(S,\bar{S})} w_{ij}$ be the sum of the weights not in the cut. As before, we define the $MAXCUT(G)$ to be the value of the Max-Cut of $G$. 
\newline \indent 
The following simple observation is needed for this subsection which follows by adding and substracting the weights within a cut. 
\begin{lemma} \label{cutdec}
Let $(S, \bar{S})$ be a cut of an edge weighted graph $G = (V,E)$ with real edge weights given by $w_{ij}$. We then have
$$w(S,\bar{S}) - \bar{w}(S,\bar{S}) = 2w(S, \bar{S}) - \sum_{(i,j) \in E} w_{ij}.$$ 
\end{lemma}
We have the following consequence of lemma \ref{cutdec}.

\begin{corollary} \label{cutnormMaxcut}
Let $G = (V_{1}, V_{2}, E)$ be the real edge weighted bipartite graph with weighted adjacency matrix given by $$B = \begin{bmatrix}
    0&A \\
    A^{T}&0
\end{bmatrix}$$ where $A$ is an $M \times M$ real matrix. Furthermore if $W(V_{1},V_{2}) = 0$ then 
$$MAXCUT(G) = \frac{1}{2}\|A\|_{\infty \to 1}.$$ Furthermore, if the columns and rows of $A$ sum to zero then
$$MAXCUT(G) = \|B \|_{C} = 2\|A\|_{C}.$$
\end{corollary}.

\begin{proof}
Let $(S, \bar{S})$ be a cut such that $w(S, \bar{S})= MAXCUT(G)$. By lemma \ref{cutdec} and the fact that $w(V_{1},V_{2}) = \sum_{(i,j)} w_{ij} = 0$, we have the following holds 
\begin{align*}
    |2W(S,\bar{S})| &= \bigg| \max_{y_{i} \in \{-1, 1 \}}
    \bigg\{\frac{1}{4}\sum_{i,j} B_{ij}(1-y_{i}y_{j}) - \frac{1}{4}\sum_{i,j} B_{ij}(1+y_{i}y_{j}) \bigg\} \bigg| \\ 
    &= \bigg|\max_{y_{i} \in \{-1, 1 \}}
    \bigg\{- \sum^{2M}_{j=M+1}\sum^{M}_{i = 1} B_{ij}y_{i}y_{j}) \bigg\} \bigg| \\
    &= \|A\|_{\infty \to 1}.
\end{align*}
If we further assume that columns and rows sum to zero then by  lemma \ref{cutnormandinf} it follows that
$w(S, \bar{S}) = 2\|A \|_{C} = \|B\|_{C}$.
\end{proof}
It was shown in proposition 3.2 of \cite{AN} in computing the cut norm $\|A\|_{C}$ is MAX-SNP hard. It follows that computing $MAXCUT(G)$ for a real weighted bipartite graph is MAX-SNP hard. We state this formally below and include an adaption of the proof in \cite{AN} for completeness. 

\begin{corollary}
If $H = (V,E)$ is an unweighted graph then there is an efficient way to construct a real weighted bipartite graph $G$ with weighted  adjacency of the form 
$$B = \begin{bmatrix}
    0&A \\
    A^{T}&0
\end{bmatrix}$$
for some real matrix $M \times M$ matrix such that $$MAXCUT(H) = \frac{1}{2}MAXCUT(G) .$$ Therefore, the problem of computing the Max-Cut of real weighted bipartite graphs is MAX-SNP hard.  
\end{corollary}

\begin{proof}
Fix an arbitrary orientation on $H$. We now define the $2|E| \times |V|$ matrix $A'$ given in \cite{AN}. For each $1 \leq i \leq m$, if the edge $e_{i}$ is oriented from $v_{j}$ to $v_{k}$, then we set $A'_{2i-1,j} = A'_{2i,k} = 1$ and $A'_{2i-1,k} = A'_{2i, j} = -1$. We set all other entries of $A'$ to zero. 
It can be checked that $\|A'\|_{C} = MAXCUT(H)$. We let $M = \max(2|E|, |V|)$ and let $A$ be the $M \times M$ block matrix, given by 
$$A = \begin{bmatrix}
     A'&0_{(2|E|-V) \times (2|E|-|V|)}
\end{bmatrix}$$ if $M = 2|E|$ and 
$$  A = \begin{bmatrix}
  A' \\ 0_{(|V|-2|E|) \times (|V|-2|E|)} 
\end{bmatrix}$$  if $M = |V|$.
Then $\|A\|_{C} = \|A'\|_{C}$.  If we let $B$ be the block matrix $$B = \begin{bmatrix}
    0&A \\
    A^{T}&0
\end{bmatrix}$$ which we can take to be weighted adjacency matrix of $G$. By corollary \ref{cutnormMaxcut}, we have that $MAXCUT(G) = \|B\|_{C} = 2\|A\|_{C}$.  Since it is known that the Max-Cut problem is MAX-SNP hard \cite{H},  our claim follows.
\end{proof}
In the work \cite{AN}, an approximation to $\|A\|_{\infty \to 1}$ was given by considering the semidefinite program \footnote{Here we follow \cite{AN}, but this problem can be shown to be equivalent to a problem in the form of the primal problem \ref{primal}.}
\begin{equation} \label{inftyonenormSDP}
    \begin{array}{ll@{}ll}
    \text{maximize} & \sum_{i,j} A_{i,j} \langle u_{i}, v_{i} \rangle  \\
    \text{subject to}& \|u_{i} \| = \|v_{j} \| = 1 \\ 
    \tag{$P_{4}$}
\end{array}
\end{equation}
where $u_{1}, \ldots, u_{M_{1}}$ and $v_{1}, \ldots, v_{M_{2}}$ are vectors in $\mathbb{R}^{M_{1} + M_{2}}$. We observe that if $\chi^{*}$ is the value of a solution to \ref{inftyonenormSDP}, then $\chi^{*} \geq \|A\|_{\infty \to 1}$. It was shown in \cite{AN} that one can obtain an approximation for $\|A\|_{\infty \to 1}$ and hence by lemma \ref{cutnormandinf} an approximation to $\|A\|_{C}$. Therefore corollary \ref{cutnormMaxcut} implies that one can approximate $MAXCUT(G)$ by using the methods in \cite{AN}. In \cite{AN}, it was shown that there exists a polynomial time randomized algorithm  that obtains $x_{i}, y_{j} \in \{-1, 1 \}$ such that  $\sum_{i,j} A_{ij}x_{i}y_{j}$ has expectation equal to $\rho\chi^{*}$ for $\rho = {2\log(1+\sqrt{2})}/{\pi}$. The procedure in \cite{AN} is based on a proof of Grothendieck's inequality which we now state \cite{Kr}. 

\begin{lemma} \label{Grothineq}
Let $A$ be an $M_{1} \times M_{2}$ real matrix and let $u_{1}, \ldots, u_{M_{1}}$ and $v_{1}, \ldots, v_{M_{2}}$ be unit vectors. Then there exist $x_{i}, y_{i} \in \{-1,1 \}$ such that 

$$\sum_{i,j} A_{i,j} \langle u_{i}, v_{i} \rangle \leq K_{G} \sum A_{i,j}x_{i}y_{j}.$$ In particular, it follows 
$$\sum_{i,j} A_{i,j} \langle u_{i}, v_{i} \rangle \leq K_{G} \|A\|_{\infty \to 1}.$$
\end{lemma}
Here $K_{G}$ is a universal constant whose value is unknown, but in \cite{Kr}, \cite{BK},  it was shown that $\sqrt{2} \leq K_{G} < {\pi}/{(2\log(1+\sqrt{2}))}$. Using this, we can show a relationship between \ref{MaxCut} and $\|A\|_{\infty \to 1}$. This relationship can be thought of as analogue of proposition \ref{Prop4.4} which was proved by showing the solution to \ref{MaxCut} is attained at $4\|A\|_{1,1}$ where we define
$\|A\|_{1,1} = \sum_{i,j} |A_{i,j}|.$
We now state the proposition. 
\begin{proposition}
    \label{Lastprop}
    Let $A$ be a real $M \times M$ matrix and let $B$ be the block matrix $$
    B = \begin{bmatrix}
    0&A\\
    A^{T}&0 .\end{bmatrix}$$  then the solution to \ref{MaxCut} is bounded above by $2K_{G} \|A\|_{\infty \to 1} + 2\|A\|_{1,1}$. Futhermore, the solution to \ref{OurSDP} is bounded by $2K_{G} \|A\|_{\infty \to 1}$. 
\end{proposition}

\begin{proof}
Let $B_{\geq 0}, B_{<0}$ be the $2M \times 2M$ matrices whose nonzero entries consist of the non-negative and negative entries of $B$ respectively. We then let $C= D_{d}(B) - B$, $C_{\geq 0} =  D_{d}(B_{\geq 0})-B_{\geq 0},$ and $C_{<0} = D_{d}(B_{<0}) - B_{<0}$. We let $G=(V_{1}, V_{2},E)$ be the real weighted bipartite graph with edge weights given by $B$. We let $G_{+}$ and $G_{-}$ be the weighted bipartite graphs with edge weights given by $B_{\geq 0}$ and $-B_{\leq 0}$ respectively. We have that $C_{\geq 0}$ is the Laplacian of $G_{+}$ and $C_{\leq 0}$ is the signless Laplacian of $G_{-}$. 
\newline \indent
As before in proposition \ref{Prop4.4}, by Slater's condition, the strong duality holds, therefore the value of the solution to \ref{MaxCut} is attained by the value of the solution to the primal problem
\begin{equation*}
\begin{array}{ll@{}ll}
    \text{maximize} & \langle C, X \rangle_{F} \\
    \text{subject to}& X_{ii} \leq 1,  &&i \in V \\
    &X \succeq 0. 
\end{array}
\end{equation*}
Since our solution $X$ is postive semidefinite, we have Gram unit vectors $v_{1}, \ldots, v_{2M}$, such that $X_{ij} = \langle v_{i}, v_{j} \rangle$. We can rewrite the objective as
\begin{align*}
    \langle C, X \rangle_{F} &= \langle C_{\geq 0}, X \rangle_{F}  + \langle C_{< 0}, X \rangle_{F} \\
    &= \sum_{(i,j) \in E, B_{ij} \geq 0} B_{ij} \|v_{i} -v_{j} \|^{2} - \sum_{(i,j) \in E, B_{ij} \leq 0} B_{ij} \|v_{i} + v_{j} \|^{2}\\
    &= 2\sum_{(i,j) \in E, B_{ij} \geq 0} B_{ij}(1- \langle v_{i},v_{j} \rangle)  - 2\sum_{(i,j) \in E, B_{ij} \leq 0} B_{ij} (1+ \langle v_{i},v_{j}\rangle )\\
    &= -2\sum_{(i,j) \in E} B_{ij}\langle v_{i},v_{j} \rangle  -4 \sum_{(i,j) \in E, B_{ij} \leq 0} B_{ij}  + 2\sum_{(i,j) \in E} B_{ij} \\
\end{align*}
where we added and subtracted $2\sum_{(i,j) \in E, B_{ij} < 0} B_{ij}(1- \langle v_{i},v_{j} \rangle)$ in the last line.  If we set $u_{1} = -v_{1}, \ldots, u_{M} = -v_{M}$ and $w_{1}= v_{M+1}, \ldots, w_{M} = v_{2M}$, we then obtain by Grothendieck's inequality, lemma  \ref{Grothineq},  that
\begin{align*}
    \langle C, X \rangle_{F}  &=  -2\sum_{(i,j) \in E} B_{ij}\langle v_{i},v_{j} \rangle  -4 \sum_{(i,j) \in E, B_{ij} \leq 0} B_{ij} + 2\sum_{(i,j) \in E} B_{ij} \\
    &= 2 \sum_{i,j} A_{ij} \langle u_{i}, w_{j} \rangle -2\sum_{i,j, B_{ij} \leq 0} B_{ij} + \sum_{i,j} B_{ij}\\
    &=   2 \sum_{i,j} A_{ij} \langle u_{i}, w_{j} \rangle + \sum_{i,j} |B_{ij}| \\
    &\leq 2K_{G}\|A \|_{\infty \to 1} + 2\|A\|_{1,1}. 
\end{align*}
\end{proof}
From here we can relate this to an upper-bound for the right hand side of proposition \ref{genfailureprob} when $C$ is obtained through the program \ref{OurSDP}. 

\begin{corollary} \label{lastcor1}
Let $A$ be a $M \times M$ real matrix and let $D$ be the diagonal matrix corresponding to the solution to the problem \ref{OurSDP}. 
Let $S_{N}$ be the sample mean defined in the adjusted algorithm for estimating $\mathrm{perm}(A)$ with samples from $\mathcal{N}(0, B+D)$.  Then for fixed $N$, we have that  
    \begin{equation}
        P(|S_{N}-\text{perm}(A)| > t) \leq \frac{3^{M}}{t^{2}N} \bigg(\frac{1}{M} (K_{G}\|A\|_{\infty \to 1}) \bigg)^{2M} .
    \end{equation}
where $K_{G}$ is Grothendieck's constant. In particular, if we take $ K/\epsilon^{2}$ samples from $\mathcal{N}(0, D+B)$ then in time $O(M/\epsilon^{2})$, this algorithm obtains an estimate of $\mathrm{perm}(A)$ within an additive error  $\epsilon (\sqrt{3}{K_{G}\|A\|_{\infty \to 1}}/{M})^{M}$ with probability at least $1-K^{-1}$.
\end{corollary}
\begin{proof}
If we let $D$ be the diagonal matrix corresponding to the solution of \ref{OurSDP}, we can obtain the following by the arithmetic-geometric mean inequality and proposition \ref{Lastprop}
\begin{align*}
     \frac{3^{M}}{t^{2}N}\prod^{2M}_{i=1}D_{ii} &\leq  \frac{3^{M}}{t^{2}N} \bigg( \frac{1}{2M} \sum^{2M}_{i=1}D_{ii} \bigg)\\
     &\leq \frac{3^{M}}{t^{2}N} \bigg(\frac{1}{M} (K_{G}\|A\|_{\infty \to 1}) \bigg)^{2M}.
\end{align*}

%ADD PROOF

%as a corollary of proposition \ref{genfailureprob}, proposition \ref{Lastprop}, and the arithmetic-geometric mean inequality.
\end{proof}

If we further assume  that $\sum_{i,j}A_{ij} = 0$ then  we can rephrase corollary \ref{lastcor1} as the following in terms of the Max-Cut problem using corollary \ref{cutnormMaxcut}.
\begin{corollary}
Let $A$ be a $M \times M$ real matrix $\sum_{i,j}A_{ij} = 0$ with and let $D$ be the diagonal matrix corresponding to the solution to the problem \ref{OurSDP}. Then our adjusted algorithm yields the following for fixed $N$,
    \begin{equation}
        P(|S_{N}-\text{perm}(A)| > t) \leq \frac{3^{M}}{t^{2}N} \bigg(\frac{1}{M} (2K_{G}MAXCUT(G)) \bigg)^{2M} 
    \end{equation}
where $G = (V_{1}, V_{2}, E)$ is the real weighted bipartite graph with weighted adjacency matrix $B$. In particular, if we take $K/\epsilon^{2}$ samples from $\mathcal{N}(0, D+B)$ then this algorithm in time $O(M/\epsilon^{2})$ obtains an estimate of $\mathrm{perm}(A)$ within an additive error  $\epsilon (2\sqrt{3}{K_{G}MAXCUT(G)}/{M})^{M}$ with probability at least $1-K^{-1}$   
\end{corollary}

\section{Discussion}
%Planar graphs
 \indent In this work we have provided a new additive estimate of $\text{perm}(A)$, the permanent of an $M \times M$ real matrix $A$ which estimates $\text{perm}(A)$ within an additive error of $\epsilon(\sqrt{3}\|A\|)^{M}$ in time polynomial in $M$.  Like \cite{G}, the accuracy of our estimate depends on $\|A\|$ the operator norm of $A$ but introduces a factor of $3$ to the deviation from the mean. However, we improve our algorithm to estimating $\text{perm}(A)$ within an additive error of for a given Gaussian embedding $C$ of $A$ in time polynomial in $M$. We find such a $C$ by the semidefinite program \ref{OurSDP} where the estimate depends on the diagonal entries of $C$ rather than the operator norm $\|A\|$. This raises the possibility to provide classes of matrices where our procedure improves over the state of the art procedure given by theorem \ref{Gurvits}. It would be satisfying to find classes of such matrices. 
 %\newline \indent 
 %One criticism of our algorithm is that it estimates $\text{perm}(A)$ up to a wide additive error. In particular, our additive estimate for the initial algorithm given in terms of the operator norm via proposition \ref{failureprob}. Though theoretically useful, this may fail to be of practical use since $\text{perm}(A) \leq \|A \|^{M}$ as shown in \cite{G} or \cite{AH}. We were able improve our algorithm, to estimate $\text{perm}(A)$ up to an additive error given in terms of diagonal entries $C_{ii}$ of our Gaussian embedding $C$ of $A$ but we fail to fully see a relationship between the quantity $\prod^{2M}_{i=1}C_{ii}$ and $\text{perm}(A)$. We give a condition in proposition \ref{genfailureprob} when our procedure provides an additive FPRAS, however, it would be interesting to see if there exists well known families of matrices that satisfy this condition. Finally, our work leaves open the question of using theorem \ref{X1...X2M} to provide a multiplicative error estimate of $\text{perm}(A)$. 
\newline \indent 
In addition, we also obtained an upper bound on the failure probability by the use of Chebyshev's inequality or more specifically, a second moment method. In particular, if we set $\eta$ to be the right hand side of proposition \ref{genfailureprob} then number of samples $N$ is polynomial in $\eta^{-1}$. We compare this to previous procedures \cite{G} and \cite{CCGP} which feature upper bounds on the failure probability where the number of samples $N$ depends log-polynomially on $\eta^{-1}$. This is through the use of a Chernoff bound or Hoeffding's inequality since the random variables considered in those works are bounded. It is likely that one could improve our estimates through the use of stronger techniques from large deviation theory or higher moment methods. 
\newline \indent Also, our work raises the question of other possible Gaussian embeddings and their potential application to this and related problems. We hope that this work will motivate further applications of Gaussian fields to estimating the permanent and other related problems.

\begin{Ack}
    We thank Noga Alon for useful discussions. In particular pointing out lemma \ref{cutdec}, Grothendieck's inequality, and helpful discussions on the real weighted Max-cut problem. We also thank Eric Vigoda and the anonymous referees for helpful suggestions on an earlier version of this work. 
\end{Ack}

\end{document}